\documentclass[12pt]{article}
\usepackage{amsmath,amssymb,xcolor,url,amsthm,enumerate,caption,subcaption,soul}
\usepackage[left=2.4cm,top=2.4cm,right=2.4cm,bottom=2.5cm]{geometry}
\newtheorem{theorem}{Theorem}[section]
\newtheorem{lemma}[theorem]{Lemma}

\newtheorem{corollary}[theorem]{Corollary}

\newtheorem{remark}[theorem]{Remark}

\DeclareMathOperator{\str}{str}
\DeclareMathOperator{\Succ}{A}
\DeclareMathOperator{\Prec}{B}

\title{A note on the strength of a hypercube}
\author{Melissa A. Huggan\thanks{Department of Mathematics, Vancouver Island University, Nanaimo, BC, Canada}, M.E. Messinger\thanks{Department of Mathematics and Computer Science, Mount Allison University, Sackville, NB, Canada}, Dylan Pearson\thanks{Department of Mathematics and Statistics, Dalhousie University, Halifax, NS, Canada}}
\date{\today}

\begin{document}
\maketitle

\begin{abstract}  As a generalization of super magic strength, the \emph{strength} of a graph was introduced in [\textsc{R. Ichishima, F.A. Muntaner-Batle, A. Oshima, Bounds for the strength of graphs, Austral. J. of Combin. 72(3) (2018) 492-508}].  For a vertex ordering $f$ of graph $G$, the strength of $f$ is the maximum sum of the labels on any pair of adjacent vertices.  The strength of $G$ is defined as the minimum strength of $f$, taken over all vertex orderings of $G$.  The strength of the hypercube is unknown, but bounded. In this note, we provide an improved upper bound for the strength of a hypercube.
\end{abstract}



\section{Introduction}

Let $G$ be a graph on $n$ vertices and $f$ be a bijective function $f: V(G) \to \left\{1,2, \ldots, n\right\}$. The \emph{strength of $f$}, denoted $\str_f(G)$, is defined as \begin{equation*}\str_f(G) = \max \{ f(u)+f(v) : uv \in E(G)\}.\end{equation*} The \emph{strength} of $G$, denoted by $\str(G)$, is defined as \begin{equation*}\str(G) = \min \{ \str_f(G) : f \text{~is a vertex ordering of~} G\}.\end{equation*}  The strength of a graph was first defined in~\cite{I.2018} and was motivated by super edge-magic labelings and the super magic strength of a graph.  See the survey~\cite{WallisEdge} for more on super edge-magic labelings and~\cite{Av,I.2018} for more on super magic strength.  Ichishima et al.\ \cite{I.2018} determined the strength of a number of graph classes, including paths, cycles, complete graphs, and complete bipartite graphs. The strength of some trees, including caterpillars and complete $n$-ary trees was determined in~\cite{Trees}. The $n$-dimensional hypercube, denoted $Q_n$, is the graph whose vertices correspond to the $(0,1)$-binary strings of length $n$; two vertices being adjacent if and only if the corresponding binary strings differ in exactly one digit.  The $n$-dimensional hypercube, $Q_n$, was considered in~\cite{I.2018}, where they proved the following.

\begin{theorem}[\cite{I.2018}]\label{Hyper} For the $n$-dimensional hypercube $Q_n$, $\str(Q_1)=3$, $\str(Q_2)=6$, and for every integer $n \geq 3$, $$2^n+n \leq \str(Q_n) \leq 2^n+2^{n-2}+1.$$\end{theorem}

For completeness, we note that the lower bound was improved in~\cite[Theorems 16, 17, and Corollary 18]{G.L.S.2021} to the following.

\begin{theorem}[\cite{G.L.S.2021}]\label{G.L.S.2021} For the $n$-dimensional hypercube $Q_n$,  $\str(Q_3) = 11$, $\str(Q_4) = 21$, $\str(Q_5) = 40$, $$\begin{array}{ll}\str(Q_n) \geq 2^n+4n-12 & \text{for $5 \leq n \leq 9$, and} \\
\str(Q_n) \geq 2^n+\lfloor \frac{n^2}{4}\rfloor +4 & \text{for $n \geq 10$.}\end{array}$$
\end{theorem}

In this note, we improve the upper bound for $\str(Q_n)$. We define a bijection $f$ between the set of all $n$-bit strings and $\{1,2,\dots,2^{n}\}$.  We first show the bijection gives the exact strength of $Q_3$, $Q_4$, and $Q_5$. We then determine the recurrence 
$$\str(Q_{n}) \leq \str_f(Q_{n}) \leq \str_f(Q_{n-2}) + 3 \cdot 2^{n-2} + \binom{n-3}{\lceil \frac{n-3}{2}\rceil}+\binom{n-2}{\lceil \frac{n-2}{2}\rceil}$$ in Corollary~\ref{cor:recurrence}. The recurrence provides an improvement over the upper bound of Theorem~\ref{Hyper} of \cite{I.2018} for $n \in \{5,6,\cdots,13\}$.  Ultimately, in Corollary~\ref{cor:last}, we show that for all $n \geq 14$, 

$$\str(Q_n) \leq 2^n+2^{n-3}+28,$$ which is an improvement of $2^{n-3}-27$ over the upper bound of Theorem~\ref{Hyper} of \cite{I.2018}. 

\section{An improved upper bound for $\mathbf{str(Q_n)}$}\label{sec:defn}

We state some important definitions before presenting the bijection between the $n$-bit strings and $\left\{1, 2, \ldots ,2^n\right\}$ and proving some initial observations.

Let $x=x_1x_2\cdots x_n$ and $y=y_1y_2\cdots y_n$ be distinct $n$-bit strings; that is, $x_i$, $y_i \in \left\{0,1\right\}$ for every $i\in \left\{1,2,\ldots , n\right\}$. The \emph{Hamming weight}, or simply \emph{weight} of a bit string is the number of non-zero bits in the string. The \emph{Hamming distance} between two $n$-bit strings $x$ and $y$ is the number of positions in which $x$ and $y$ differ. Consequently, two vertices of $Q_n$ are adjacent if and only if the Hamming distance between the corresponding $n$-bit strings is one.  

We define $x < y$ to be in {\it lexicographic order} if $x_i=y_i$ for $i<k$ and $x_k < y_k$ for some $k \leq n$.  We define \emph{reverse lexicographic order} as the reverse ordering of the bit strings.  That is, if $x<y<z$ then the sequence $(x,y,z)$ is in lexicographic order and the sequence $(z,y,x)$ is in reverse lexicographic order.  We will often refer to the lexicographic or reverse lexicographic ordering of the set of bit strings of a fixed length and weight, so we define $S_n^i$ to be the sequence of $n$-bit strings of weight $i$, taken in lexicographic order; and $R_n^i$ be the sequence of $n$-bit strings of weight $i$, taken in reverse lexicographic order. Finally, define $S_n$ to be the sequence: 
$$S_n = \big(R_n^1,R_n^3,R_n^5,\dots,R_n^{n-1},S_n^n,S_n^{n-2},\dots,S_n^4,S_n^2,S_n^0\big) ~~\text{for even $n$;}$$ $$S_n = \big(R_n^1,R_n^3,R_n^5,\dots,R_n^{n-2},R_n^n,S_n^{n-1},\dots,S_n^4,S_n^2,S_n^0\big) ~~\text{for odd $n$.}$$

Observe that $S_n^n = R_n^n$ since both contain only the string of weight $n$ (i.e.\ the all $1$s string).  Since every $n$-bit string appears exactly once in $S_n$ we can define a bijection $f$ between the sequence of all $n$-bit strings (in $S_n$) and $\{1,2,\dots,2^n\}$. Note that the first $2^{n-1}$ strings in $S_n$ each have odd weight and the second $2^{n-1}$ strings each have even weight. As an example, the sequences $S_n$ for $n = 3,4,5,6$ are given in Tables~\ref{tab:n=3}, \ref{tab:n=4}, \ref{tab:n=5}, \ref{tab:n=6}, respectively. Although it is already known that $\str(Q_3)=11$, $\str(Q_4)=21$ and $\str(Q_5) = 40$, the reader can verify from Tables~\ref{tab:n=3}, \ref{tab:n=4}, and~\ref{tab:n=5} that these values are achieved using the bijection provided by $S_n$.

\medskip

The \emph{one's complement} of a binary string is the string obtained by flipping all the bits in the representation, for example, $1011$ is the one's complement of $0100$.  We will represent the one's complement of a string $x$ as $\overline{x}$.  For odd $n$, suppose $x$ is an $n$-bit string with weight $i$ for some odd positive integer $i$.  Then $x \in R_n^i$ and $\overline{x} \in S_n^{n-i}$ and by construction of $S_n$, \begin{equation}\label{eqn:oddn}f(x) + 2^{n-1} = f(\overline{x}).\end{equation}

\begin{table}[h!]
\begin{minipage}{0.45\textwidth}
\begin{subtable}[t]{0.9\textwidth}
 \centering
 \begin{tabular}{|c|c||c|c|}
\hline
\text{bit string $x$} & $f(x)$ & \text{bit string $x$} & $f(x)$ \\
\hline
100 & 1 & 011 & 5 \\
010 & 2 & 101 & 6 \\
001 & 3 & 110 & 7 \\
\hline
111 & 4 & 000 & 8 \\
\hline
\end{tabular}
 \caption{Sequence $S_3$ with $R_3^1$, $R_3^3$ on the left, $S_3^2$, $S_3^0$ on the right.}
\label{tab:n=3}
\end{subtable}
\bigskip

\begin{subtable}{0.9\textwidth}
 \centering
\begin{tabular}{|c|c|}
\hline
\text{bit string $x$} & $f(x)$ \\
\hline
1000 & 1 \\
0100 & 2 \\
0010 & 3 \\
0001 & 4 \\
\hline
1110 & 5 \\
1101 & 6 \\
1011 & 7 \\
0111 & 8 \\
\hline
\end{tabular}\begin{tabular}{|c|c|}
\hline
\text{bit string $x$} & $f(x)$  \\
\hline
1111 & 9 \\
\hline
0011 & 10 \\
0101 & 11 \\
0110 & 12 \\
1001 & 13 \\
1010 & 14 \\
1100 & 15 \\
\hline
0000 & 16 \\
\hline
\end{tabular}
 \caption{Sequence $S_4$ with $R_4^1$, $R_4^3$ on the left, $S_4^4$, $S_4^2$, $S_4^0$ on the right.}
\label{tab:n=4}
\end{subtable}

\bigskip
 \begin{subtable}{0.9\textwidth}
 \centering
\begin{tabular}{|c|c||c|c|}
\hline
\text{bit string $x$} & $f(x)$ & \text{bit string $x$} & $f(x)$ \\
\hline
10000 & 1 & 01111 & 17 \\
01000 & 2 & 10111 & 18\\
00100 & 3 & 11011 & 19\\
00010 & 4 & 11101 & 20\\
00001 & 5 & 11110 & 21\\
\hline
11100 & 6 & 00011 & 22\\
11010 & 7 & 00101 & 23\\
11001 & 8 & 00110 & 24\\
10110 & 9 & 01001 & 25\\
10101 & 10 & 01010 & 26\\
10011 & 11 & 01100 & 27\\
01110 & 12 & 10001 & 28\\
01101 & 13 & 10010 & 29\\
01011 & 14 & 10100 & 30\\
00111 & 15 & 11000 & 31\\
\hline
11111 & 16 & 00000 & 32\\ 
\hline
\end{tabular}
\caption{Sequence $S_5$ with $R_5^1$, $R_5^3$, $R_5^5$ on the left, $S_5^4$, $S_5^2$, $S_5^0$ on the right.}
\label{tab:n=5}
\end{subtable}
\end{minipage}~~~~
\begin{minipage}{0.45\textwidth}
 \begin{subtable}[t]{0.9\textwidth}
 \centering
\begin{tabular}{|c|c|}
\hline
\text{bit string $x$} & $f(x)$ \\
\hline
100000 & 1 \\
010000 & 2 \\
001000 & 3 \\
000100 & 4 \\
000010 & 5 \\
000001 & 6 \\
\hline
111000 & 7 \\
110100 & 8\\
110010 & 9\\
110001 & 10\\
101100 & 11\\
101010 & 12\\
101001 & 13\\
100110 & 14\\
100101 & 15\\
100011 & 16\\
011100 & 17\\
011010 & 18 \\
011001 & 19\\
010110 & 20\\
010101 & 21\\
010011 & 22\\
001110 & 23\\
001101 & 24\\
001011 & 25\\
000111 & 26\\
\hline
111110 & 27\\
111101 & 28\\
111011 & 29\\
110111 & 30\\
101111 & 31\\
011111 & 32\\
\hline
\end{tabular}\begin{tabular}{||c|c|}
\hline
\text{bit string $x$} & $f(x)$ \\
\hline
111111 & 33\\
\hline
001111 & 34\\
010111 & 35\\
011011 & 36\\
011101 & 37\\
011110 & 38\\
100111 & 39\\
101011 & 40\\
101101 & 41\\
101110 & 42\\
110011 & 43\\
110101 & 44\\
110110 & 45\\
111001 & 46\\
111010 & 47\\
111100 & 48\\
\hline
000011 & 49\\
000101 & 50\\
000110 & 51\\
001001 & 52\\
001010 & 53\\
001100 & 54\\
010001 & 55\\
010010 & 56\\
010100 & 57\\
011000 & 58\\
100001 & 59\\
100010 & 60\\
100100 & 61\\
101000 & 62\\
110000 & 63\\
\hline
000000 & 64\\
\hline
\end{tabular} 
\caption{Sequence $S_6$ with $R_6^1$, $R_6^3$, $R_6^5$ on the left, $S_6^6$, $S_6^4$, $S_6^2$, $S_6^0$ on the right.}
\label{tab:n=6}
\end{subtable} 

\vspace{0.9in}
~

\end{minipage}
\caption{}
\end{table}

Finally, we use the notation $0x$ or $x0$ to indicate that we prefix $0$ to $x$ or append $0$ to $x$, respectively.  

\subsection{Preliminary observations}

\begin{remark}\label{remark1}
For any integer $n \geq 1$, let $x$ and $y$ be two $n$-bit strings. The vertices of $Q_n$ that correspond to $x$ and $y$ are adjacent in $Q_n$ if and only if $x$ and $y$ have Hamming distance one. Fixing a bijective function $f$ from the $n$-bit strings to $\{1, 2, \ldots, 2^n\}$ necessarily implies that there exist $n$-bit strings $x$ and $y$ such that $\str_f(Q_n) = f(x)+f(y)$. 

For a pair $x,y$ of $n$-bit strings for which  $\str_f(Q_n)=f(x)+f(y)$ and $x$ has weight $i$, we make the following assumption throughout this note: 

\begin{center}
    $x$ has odd weight $i$ and $y$ has even weight $i-1$.
\end{center}

\noindent
Since $x$ and $y$ have Hamming distance one, their weights will be of different parity; thus we may assume $i$ is odd.  Consequently, $y$ must have even weight $i-1$ or $i+1$.  Strings of weight $i-1$ fall later in sequence $S_n$ (and therefore map to larger function values) than strings of weight $i+1$. Thus, if $x$ has odd weight $i$ then $y$ has even weight $i-1$.
 \end{remark}

\begin{lemma}\label{lem:ends} For any integer $n \geq 2$, there exists an $(n-1)$-bit string $w$ such that $x=w1$, $y=w0$ and $\str_f(Q_n) = f(x)+f(y)$.
\end{lemma}

\begin{proof} Let $x$ and $y$ be $n$-bit strings for which $\str_f(Q_n) = f(x)+f(y)$.  Consequently, $x$ and $y$ differ only in one digit. Assume, by Remark~\ref{remark1}, that $x$ has odd weight $i$ and $y$ has weight $i-1$.    We first prove the last digit in $y$ will be $0$, regardless of the last digit in $x$.

First, suppose the last digit in $x$ is $0$.  For a contradiction, assume the last digit in $y$ is $1$.  As $x$ and $y$ differ only in the last digit, $y$ has a larger weight than $x$, which provides a contradiction.  Consequently, if the last digit in $x$ is $0$ then the last digit in $y$ is $0$.

Second, suppose the last digit of $x$ is $1$, so $x=z1$ for some $(n-1)$-bit string $z$ of weight $i-1$. Of the $n$-bit strings of weight $i-1$ that differ from $x$ in exactly one digit, $z0$ has the largest binary representation. Since the subsequence $S_n^{i-1}$ in $S_n$ is in lexicographic order, $f(z0) \geq f(y)$.  Since $f(x)+f(y)$ is maximum, $y=z0$.  Consequently, if the last digit in $x$ is $1$ then the last digit in $y$ is $0$.

Therefore, if $\str_f(Q_n)=f(x)+f(y)$ and $x$ has odd weight, then the last digit of $y$ is $0$.  Thus, $y = w0$ for some $(n-1)$-bit string $w$ of weight $i-1$.  Of the $n$-bit strings of weight $i$ that differ from $y$ in exactly one digit, $w1$ has the smallest binary representation.  Since the subsequence $R_n^i$ in $S_n$ is in reverse lexicographic order, $f(w1) \geq f(x)$.  Since $f(x)+f(y)$ is maximum, $x=w1$.\end{proof} 

We state the next observation for $n$-bit strings of weights $i$ and $i-1$, but apply it to strings of other lengths and weights as well. For an $n$-bit string $x$ of weight $i$, we define $\Succ(x,R_n^i)$ to be the set of strings in $R_n^i$ that come after $x$.  For an $n$-bit string $y$ of  weight $i-1$, we define $\Prec(y,S_n^{i-1})$ to be the set of strings in $S_n^{i-1}$ that come before $y$.

Throughout the rest of the paper we will use the following binomial identities to simplify expressions. For $n, k \in \mathbb{Z}$, $n > k > 0$, 
\begin{equation}\label{Eq: Pascal's Triangle}
    \tbinom{n}{k} = \tbinom{n-1}{k-1} + \tbinom{n-1}{k}.
\end{equation}

For $n,k \in \mathbb{Z}$, $n\geq k \geq 0$, 
\begin{equation}\label{Eq: Symmetry}
    \tbinom{n}{k} = \tbinom{n}{n-k}.
\end{equation}

\begin{lemma}\label{obs:x'y'} For any integer $n \geq 2$, let $w$ be an $(n-1)$-bit string of weight $i-1$ for some odd integer $i$.  Then $$f(w1)+f(w0) = 2^n + \tbinom{n-1}{i} - |\Succ(w1,R_n^i)|+|\Prec(w0,S_n^{i-1})|+1.$$\end{lemma}

\begin{proof} For any integer $n \geq 2$, let $w$ be an $(n-1)$-bit string of weight $i-1$ for some odd integer $i$. Since $w1 \in R_n^i$ and $w0 \in S_n^{i-1}$, we find $$f(w1) = |R_n^1|+|R_n^3|+\cdots+|R_n^i| - |\Succ(w1,R_n^i)| = \tbinom{n}{1}+ \tbinom{n}{3}+\cdots+\tbinom{n}{i}-|\Succ(w1,R_n^i)|$$ and $$f(w0) = 2^n - |S_n^0|-|S_n^2|-\cdots - |S_n^{i-1}| +|\Prec(w0,S_n^{i-1})|+1 = 2^n - \tbinom{n}{0}-\tbinom{n}{2}-\cdots-\tbinom{n}{i-1}+|\Prec(w0,S_n^{i-1})|+1.$$
To simplify the sum $f(w1)+f(w0)$, we use the identity\footnote{This can be proved by induction on $\ell$ and using the identity~(\ref{Eq: Pascal's Triangle}); as it is an undergraduate exercise, we omit the proof here.} $\sum_{k=0}^\ell (-1)^k \binom{N}{k} = (-1)^\ell \binom{N-1}{\ell}$ 
\begin{eqnarray*}f(w1)+f(w0) &=& 2^n- \Big( \sum_{k=0}^i (-1)^k \tbinom{n}{k}\Big) - |\Succ(w1,R_n^i)|+|\Prec(w0,S_n^{i-1})|+1 \\ &=& 2^n + \tbinom{n-1}{i} - |\Succ(w1,R_n^i)|+|\Prec(w0,S_n^{i-1})|+1.\end{eqnarray*}  \end{proof}

\subsection{Main results}\label{ss: Main}

We now provide an improved upper bound on the strength of the hypercube. Since $\str(Q_n)$ is known for $n \in \{1,2,3,4\}$, we consider $n \geq 5$ in this subsection.

\begin{theorem}\label{thm:both0} For any integer $n \geq 5$, let $x$ and $y$ be $n$-bit strings for which $\str_f(Q_{n})=f(x)+f(y)$.  If $x$ and $y$ both begin with $0$ then $$\str_f(Q_{n}) \leq \str_f(Q_{n-2}) + 3 \cdot 2^{n-2} + \binom{n-3}{\lceil \frac{n-3}{2}\rceil}+\binom{n-2}{\lceil \frac{n-2}{2}\rceil}.$$\end{theorem} 

\begin{proof} Let $x,y$ be $n$-bit strings for which $\str_f(Q_{n}) = f(x)+f(y)$, and assume $x$ and $y$ both begin with $0$.  By Remark~\ref{remark1}, $x$ has odd weight $i$ and $y$ has even weight $i-1$.  By Lemma~\ref{lem:ends}, $x$ and $y$ differ only in that the last bit in $x$ is $1$ and the last bit in $y$ is $0$.  Therefore, $x = 0w1$ and $y=0w0$ for some $(n-2)$-bit string $w$ of weight $i-1$. We apply Lemma~\ref{obs:x'y'} to both $f(x)+f(y)$ and $f(w1)+f(w0)$:   
\begin{multline}\label{eqqq} \hspace{-0.1in}\Big(f(x) + f(y)\Big) - \Big(f(w1)+f(w0)\Big) \hfill  \\
=2^{n}-2^{n-1}+\tbinom{n-1}{i}-\tbinom{n-2}{i}-|\Succ(x,R_{n}^i)|+|\Prec(y,S_{n}^{i-1})|+|\Succ(w1,R_{n-1}^i)|-|\Prec(w0,S_{n-1}^{i-1})| \hfill \\
=2^{n-1}+\tbinom{n-2}{i-1}-|\Succ(0w1,R_{n}^i)|+|\Prec(0w0,S_{n}^{i-1})|+|\Succ(w1,R_{n-1}^i)|-|\Prec(w0,S_{n-1}^{i-1})|. \hfill
\end{multline}
Observe that $R_{n}^i$ consists of subsequence $R_{n-1}^{i-1}$ with $1$ prefixed to each string; followed by subsequence $R_{n-1}^{i}$, with $0$ prefixed to each string.  Consequently, the strings that succeed $x=0w1$ in $R_{n}^i$ and the strings that succeed $w1$ in $R_{n-1}^i$ are the same except that those in $R_{n}^i$ that have a $0$ prefixed.  Therefore, \begin{equation}\label{eqn:succ}|\Succ(0w1,R_{n}^i)|=|\Succ(w1,R_{n-1}^i)|.\end{equation}

Similarly, $S_{n}^{i-1}$ consists of subsequence $S_{n-1}^{i-1}$ with 0 prefixed to each string; followed by subsequence $S_{n-1}^{i-2}$ with $1$ prefixed to each string.  Thus, \begin{equation}\label{eqn:prec}|\Prec(0w0,S_{n}^{i-1})|=|\Prec(w0,S_{n-1}^{i-1})|.\end{equation}

Using Equations (\ref{eqn:succ})-(\ref{eqn:prec}), the expression from (\ref{eqqq}) simplifies to $$\Big(f(x) + f(y)\Big) - \Big(f(w1)+f(w0)\Big) = 2^{n-1}+\tbinom{n-2}{i-1} \leq 2^{n-1} + \tbinom{n-2}{\lceil (n-2)/2 \rceil}.$$   Finally, \begin{eqnarray}\str_f(Q_{n}) &=& f(x)+f(y)\nonumber \\ &\leq& f(w1)+f(w0)+2^{n-1}+\tbinom{n-2}{\lceil (n-2)/2 \rceil}\nonumber \\ &\leq& \str_f(Q_{n-1})+2^{n-1}+\tbinom{n-2}{\lceil (n-2)/2\rceil} \label{YA}\\
&\leq& \Big[\str_f(Q_{n-2})+2^{n-2}+\tbinom{n-3}{\lceil (n-3)/2 \rceil}\Big] +2^{n-1}+\tbinom{n-2}{\lceil (n-2)/2\rceil}\nonumber \\
&=& \str_f(Q_{n-2}) + 3 \cdot 2^{n-2}+\tbinom{n-3}{\lceil(n-3)/2 \rceil}+\tbinom{n-2}{\lceil (n-2)/2\rceil}.\nonumber
\end{eqnarray}\end{proof}

Note that the theorem statement is weaker than the result of Inequality~(\ref{YA}). We do this to ensure that upper bounds for all cases in this section match. 

\begin{theorem}\label{thm:both1odd} For any integer $n \geq 5$, let $x$ and $y$ be $n$-bit strings for which $\str_f(Q_{n}) = f(x)+f(y)$.  If $x,y$ both begin with $1$ and $n$ is odd, then $$\str_f(Q_{n}) \leq \str_f(Q_{n-2}) + 3 \cdot 2^{n-2} + \binom{n-3}{\lceil \frac{n-3}{2}\rceil}+\binom{n-2}{\lceil \frac{n-2}{2}\rceil}.$$
\end{theorem}

\begin{proof} Let $x$ and $y$ be $n$-bit strings for which $str_f(Q_n) = f(x)+f(y)$, and assume $n$ is odd and $x$ and $y$ both begin with $1$. By Remark~\ref{remark1}, $x$ has odd weight $i$ and $y$ has even weight $i-1$.  By Lemma~\ref{lem:ends}, $x$ and $y$ differ only in that the last bit in $x$ is $1$ and the last bit in $y$ is $0$.  Therefore, $x = 1w1$ and $y=1w0$ for some $(n-2)$-bit string $w$ of weight $i-2$. 

Then $x = 1w1 \in R_{n}^i$, and since $n$ is odd, by Equation~(\ref{eqn:oddn}),  
$$f(x)+2^{n-1}=f(1w1) +2^{n-1} = f(\overline{x}) ~~ \Longleftrightarrow ~~ f(1w1)+2^{n-1} = f(0\overline{w}0).$$  Similarly, $y = 1w0 \in S_{n}^{i-1}$, and so $\overline{y} = 0\overline{w}1 \in R_{n}^{n-i+1}$.  As $n-i+1$ and $n$ are both odd, $$f(0\overline{w}1) + 2^{n-1} = f(1w0).$$  So $$\str_f(Q_{n}) = f(x)+f(y) = f(1w1)+f(1w0) = f(0\overline{w}1)+f(0\overline{w}0).$$

Observe that $0\overline{w}1$ is an $n$-bit string of odd weight $n-i+1$ and $0\overline{w}0$ is an $n$-bit string of even weight $n-i$. Since $\str_f(Q_{n})=f(0\overline{w}1)+f(0\overline{w}0)$, by Theorem~\ref{thm:both0},  $$\str_f(Q_{n}) \leq \str_f(Q_{n-2}) + 3 \cdot 2^{n-2} + \tbinom{n-3}{\lceil (n-3)/2\rceil}+\tbinom{n-2}{\lceil (n-2)/2\rceil}.$$\end{proof}

\begin{theorem}\label{thm:both11even} For any integer $n \geq 5$, let $x$ and $y$ be $n$-bit strings for which $\str_f(Q_{n})=f(x)+f(y)$. If $x$ and $y$ both begin with $11$ and $n$ is even, then $$\str_f(Q_{n}) \leq \str_f(Q_{n-2}) + 3 \cdot 2^{n-2} + \binom{n-3}{\lceil \frac{n-3}{2} \rceil}+\binom{n-2}{\lceil \frac{n-2}{2}\rceil}.$$ \end{theorem} 

\begin{proof} Let $x,y$ be $n$-bit strings for which $\str_f(Q_{n}) = f(x)+f(y)$ and assume $n$ is even, and $x$ and $y$ both begin with $11$.  By Remark~\ref{remark1}, $x$ has odd weight $i$ and $y$ has even weight $i-1$.  By Lemma~\ref{lem:ends}, $x$ and $y$ differ only in that the last bit in $x$ is $1$ and the last bit in $y$ is $0$.  Therefore, $x = 11w1$ and $y=11w0$ for some string $w$ of length $n-3$ and weight $i-3$. We apply Lemma~\ref{obs:x'y'} to both $f(x)+f(y)$ and $f(w1)+f(w0)$: 
$$f(x)+f(y) = 2^{n}+ \tbinom{n-1}{i} - |\Succ(x,R_{n}^i)|+|\Prec(y,S_{n}^{i-1})|+1,$$
$$f(w1)+ f(w0) = 2^{n-2} + \tbinom{n-3}{i-2}-|\Succ(w1,R_{n-2}^{i-2})|+|\Prec(w0,S_{n-2}^{i-3})|+1,$$ and so \begin{multline}\Big(f(x)+f(y)\Big) - \Big((f(w1)+f(w0)\Big) ~=~ 3 \cdot 2^{n-2} + \tbinom{n-1}{i}-\tbinom{n-3}{i-2} \\ -|\Succ(x,R_{n}^i)|+|\Succ(w1,R_{n-2}^{i-2})|+|\Prec(y,S_{n}^{i-1})|-|\Prec(w0,S_{n-2}^{i-3})|.\label{eqThirdOne}\end{multline}

To simplify Equation~(\ref{eqThirdOne}), we first relate $|\Succ(x,R_{n}^i)|$ to $|\Succ(w1,R_{n-2}^{i-2})|$.  Observe that $R_{n}^i$ consists of:
\begin{itemize} 
\item subsequence $R_{n-2}^{i-2}$ with $11$ prefixed to each string; followed by\vspace{-0.09in}
\item subsequence $R_{n-2}^{i-1}$, with $10$ prefixed to each string, followed by \vspace{-0.09in}
\item subsequence $R_{n-2}^{i-1}$, with $01$ prefixed to each string, followed by \vspace{-0.09in} 
\item subsequence $R_{n-2}^i$, with $00$ prefixed to each string.\end{itemize} Thus, \begin{equation}\label{eqFirstOne}|\Succ(x,R_{n}^i)| = |\Succ(11w1,R_{n}^i)| =  |\Succ(w1,R_{n-2}^{i-2})| + 2\tbinom{n-2}{i-1}+\tbinom{n-2}{i}.\end{equation}

We next relate $|\Prec(y,S_{n}^{i-1})|$ to $|\Prec(w0,S_{n-2}^{i-3})|$.  
 Observe that $S_{n}^{i-1}$ consists of:
\begin{itemize} 
\item subsequence $S_{n-2}^{i-1}$ with $00$ prefixed to each string; followed by \vspace{-0.09in}
\item subsequence $S_{n-2}^{i-2}$, with $01$ prefixed to each string, followed by\vspace{-0.09in}
\item subsequence $S_{n-2}^{i-2}$, with $10$ prefixed to each string, followed by\vspace{-0.09in} 
\item subsequence $S_{n-2}^{i-3}$, with $11$ prefixed to each string.\end{itemize} Thus, \begin{equation}\label{eqSecondOne} |\Prec(y,S_{n}^{i-1})| = |\Prec(11w1,S_{n}^{i-1})|= \tbinom{n-2}{i-1}+ 2\tbinom{n-2}{i-2}+|\Prec(w0,S_{n-2}^{i-3})|.\end{equation}

We now use Equations~(\ref{eqFirstOne}) and~(\ref{eqSecondOne}), along with Equation~(\ref{Eq: Pascal's Triangle}), to simplify Equation~(\ref{eqThirdOne}) to \begin{eqnarray*}\Big(f(x)+f(y)\Big) - \Big(f(w1)+f(w0)\Big) &=& 3 \cdot 2^{n-2} + \tbinom{n-1}{i}-\tbinom{n-3}{i-2} -\tbinom{n-2}{i-1}-\tbinom{n-2}{i}+2\tbinom{n-2}{i-2} \\ &=& 3 \cdot 2^{n-2} - \tbinom{n-3}{i-2}+2\tbinom{n-2}{i-2} \\ &=& 3 \cdot 2^{n-2} + \tbinom{n-3}{i-3}+\tbinom{n-2}{i-2}  \\ & \leq& 3 \cdot 2^{n-2} + \tbinom{n-3}{\lceil (n-3)/2\rceil}+\tbinom{n-2}{\lceil (n-2)/2 \rceil}.\end{eqnarray*} Thus, \begin{eqnarray*}\str_f(Q_{n}) = f(x)+f(y) &\leq& f(w1)+f(w0) + 3\cdot 2^{n-2}+ \tbinom{n-3}{\lceil (n-3)/2\rceil}+\tbinom{n-2}{\lceil (n-2)/2\rceil } \\ &\leq& \str_f(Q_{n-2})+ 3\cdot 2^{n-2}+ \tbinom{n-3}{\lceil (n-3)/2\rceil}+\tbinom{n-2}{\lceil (n-2)/2\rceil }.\end{eqnarray*}\end{proof}

What remains to consider is the case where $x$ and $y$ both begin with $10$ and $n$ is even; and the following lemma will be used. The lemma relates subsequence $S_{n-2}^j$ to subsequence $S_{n-2}^{n-j-2}$ (and also to subsequences $R_{n-2}^j$ and $R_{n-2}^{n-j-2}$).  This will be helpful in relating the number of $(n-2)$-bit strings that precede (and succeed) a particular $(n-2)$-bit string to the number of $(n-2)$-bit strings that precede (and succeed) its one's complement.

\begin{lemma}\label{lemma0.1} Let $n-2$ be an even positive integer and $j: 0\leq j \leq n-2$, with $j \neq \frac{n-2}{2}$ and $S_{n-2}^j = (v_1,v_2,\dots,v_{k-1},v_k)$, where $k = \binom{n-2}{j}$. Then \begin{itemize}
\item $S_{n-2}^{n-j-2} = (\overline{v_k},\overline{v_{k-1}},\dots,\overline{v_2},\overline{v_1})$;
\item $R_{n-2}^j = (v_k,v_{k-1},\dots,v_2,v_1)$; and 
\item $R_{n-2}^{n-j-2}=(\overline{v_1},\overline{v_2},\dots,\overline{v_{k-1}},\overline{v_k}).$\end{itemize}\end{lemma}

\begin{proof} 
Let $v_a,v_b \in S_{n-2}^j$ where $v_a < v_b$. Suppose $v_a$ and $v_b$ agree in the first $i-1$ digits, for some integer $i: 0 \leq i-1 \leq n-3$ and disagree in the $i$th digit. Since $v_a < v_b$, the $i$th digits of $v_a$ and $v_b$ are $0$ and $1$, respectively. Thus, $\overline{v_b} < \overline{v_a}$.  Applying the above argument to every pair $v_a$, $v_b$ for which $v_a < v_b$, the conclusion that $S_{n-2}^{n-j-2} = (\overline{v_k},\overline{v_{k-1}},\dots,\overline{v_2},\overline{v_1})$ follows.

If $S_{n-2}^j = (v_1,v_2,\dots,v_{k-1},v_k)$ then $R_{n-2}^j = (v_k,v_{k-1},\dots,v_2,v_1)$.  A similar argument shows $R_{n-2}^{n-j-2} = (\overline{v_1},\overline{v_2},\dots,\overline{v_{k-1}},\overline{v_k})$. \end{proof}

\begin{theorem}\label{thm:both10even} For any integer $n \geq 5$, let $x$ and $y$ be $n$-bit strings for which $\str_f(Q_{n}) = f(x)+f(y)$.  If $x$ and $y$ both begin with $10$ and $n$ is even, then $$\str_f(Q_{n}) \leq \str_f(Q_{n-2}) + 3 \cdot 2^{n-2} + \binom{n-3}{\lceil \frac{n-3}{2}\rceil}+\binom{n-2}{\lceil \frac{n-2}{2}\rceil}.$$ \end{theorem}

\begin{proof} Let $x$ and $y$ be $n$-bit strings for which $\str_f(Q_{n}) = f(x)+f(y)$, and assume $n$ is even, and $x$ and $y$ both begin with $10$.  By Remark~\ref{remark1}, $x$ has weight $i$ and $y$ has weight $i-1$.  By Lemma~\ref{lem:ends}, $x$ and $y$ differ only in that the last bit in $x$ is $1$ and the last bit in $y$ is $0$.  Thus, $x = 10w1$ and $y=10w0$ for some string $w$ of length $n-3$ and weight $i-2$.  We will bound the difference $(f(x)+f(y)) - (f(\overline{w}1)+f(\overline{w}0))$ because we can exploit the relationships between $f(x)$ and $f(\overline{w}0)$ and between $f(y)$ and $f(\overline{w}1)$.  

By Lemma~\ref{obs:x'y'}, $$f(x)+f(y) = 2^{n}+\tbinom{n-1}{i}-|\Succ(x,R_{n}^i)|+|\Prec(y,S_{n}^{i-1})|+1$$ and since $\overline{w}$ is an $(n-3)$-bit string of weight $n-i-1$, by the same lemma: $$f(\overline{w}1)+f(\overline{w}0)=2^{n-2} + \tbinom{n-3}{n-i}-|\Succ(\overline{w}1,R_{n-2}^{n-i})|+|\Prec(\overline{w}0,S_{n-2}^{n-i-1})|+1.$$  Thus,\medskip 

\noindent $\Big(f(x)+f(y)\Big) - \Big(f(\overline{w}1)+f(\overline{w}0)\Big)$ $$=3 \cdot 2^{n-2}+\tbinom{n-1}{i}-\tbinom{n-3}{n-i} -|\Succ(x,R_{n}^i)|+|\Prec(y,S_{n}^{i-1})| +|\Succ(\overline{w}1,R_{n-2}^{n-i})|-|\Prec(\overline{w}0,S_{n-2}^{n-i-1})|$$ \begin{equation} =3 \cdot 2^{n-2} + \tbinom{n-1}{i}-\tbinom{n-3}{i-3}
-|\Succ(x,R_{n}^i)|+|\Prec(y,S_{n}^{i-1})|+|\Succ(\overline{w}1,R_{n-2}^{n-i})|-|\Prec(\overline{w}0,S_{n-2}^{n-i-1})|.
\label{eqtogether} \end{equation}

\bigskip

Our goal is to simplify Equation~(\ref{eqtogether}).  We first observe that \begin{equation}\label{eq2nd} |\Prec(\overline{w}0,S_{n-2}^{n-i-1})|+|\Succ(\overline{w}0,S_{n-2}^{n-i-1})|+1 = \tbinom{n-2}{n-i-1} = \tbinom{n-2}{i-1}.\end{equation}
We next relate $|\Succ(x,R_{n}^i)|$ to $|\Prec(\overline{w}0,S_{n-2}^{n-i-1})|$.  The sequence $R_{n}^i$ consists of \begin{itemize}
\item sequence $R_{n-2}^{i-2}$ with $11$ prefixed to each string in the sequence; followed by\vspace{-0.1in}
\item sequence $R_{n-2}^{i-1}$ with $10$ prefixed to each string in the sequence; followed by\vspace{-0.1in}
\item sequence $R_{n-2}^{i-1}$ with $01$ prefixed to each string in the sequence; followed by\vspace{-0.1in}
\item sequence $R_{n-2}^i$ with $00$ prefixed to each string in the sequence.
\end{itemize}
Then \begin{eqnarray} |\Succ(x,R_{n}^i)| &=& |\Succ(10w1,R_{n}^i)| \nonumber \\ &=& |\Succ(w1,R_{n-2}^{i-1})|+\tbinom{n-2}{i-1}+\tbinom{n-2}{i} \nonumber\\ 
&=& |\Succ(\overline{w}0,S_{n-2}^{n-i-1})|+\tbinom{n-2}{i-1}+\tbinom{n-2}{i} ~~~\text{by Lemma~\ref{lemma0.1}}\nonumber \\ &=& 2\tbinom{n-2}{i-1}+\tbinom{n-2}{i}-1-|\Prec(\overline{w}0,S_{n-2}^{n-i-1})| ~~~\text{by Equation~(\ref{eq2nd}).}\label{eq1st}\end{eqnarray}  Equation~(\ref{eq1st}) establishes the relationship between $|\Succ(x,R_{n}^i)|$ and $|\Prec(\overline{w}0,S_{n-2}^{n-i-1})|$.

We now observe that \begin{equation} |\Succ(\overline{w}1,R_{n-2}^{n-i})|+|\Prec(\overline{w}1,R_{n-2}^{n-i})|+1 = \tbinom{n-2}{n-i} = \tbinom{n-2}{i-2}\label{eqzz} \end{equation} and we next relate $|\Prec(y,S_{n}^{i-1})|$ to $|\Succ(\overline{w}1,R_{n-2}^{n-i})|$. The sequence $S_{n}^{i-1}$ consists of \begin{itemize}
\item sequence $S_{n-2}^{i-1}$ with $00$ prefixed to each string in the sequence; followed by\vspace{-0.1in}
\item sequence $S_{n-2}^{i-2}$ with $01$ prefixed to each string in the sequence; followed by\vspace{-0.1in}
\item sequence $S_{n-2}^{i-2}$ with $10$ prefixed to each string in the sequence; followed by\vspace{-0.1in}
\item sequence $S_{n-2}^{i-3}$ with $11$ prefixed to each string in the sequence.
\end{itemize}

Then \begin{eqnarray} |\Prec(y,S_{n}^{i-1})| &=& |\Prec(10w0,S_{n}^{i-1})| \nonumber \\  
&=& |\Prec(w0,S_{n-2}^{i-2})|+\tbinom{n-2}{i-2}+\tbinom{n-2}{i-1} \nonumber \\
&=& |\Prec(\overline{w}1,R_{n-2}^{n-i})|+\tbinom{n-2}{i-2}+\tbinom{n-2}{i-1} ~~~\text{by Lemma~\ref{lemma0.1}}\nonumber \\
&=& \tbinom{n-2}{i-2}-|\Succ(\overline{w}1,R_{n-2}^{n-i})|-1+\tbinom{n-2}{i-2}+\tbinom{n-2}{i-1} ~~~\text{by Equation~(\ref{eqzz}) \nonumber}\\
&=& 2\tbinom{n-2}{i-2}+\tbinom{n-2}{i-1}-1-|\Succ(\overline{w}1,R_{n-2}^{n-i})|.\label{eqczz}
\end{eqnarray}

We use Equations~(\ref{eq1st}) and~(\ref{eqczz}) to simplify the expression in Equation~(\ref{eqtogether}) to:
\begin{eqnarray*} \Big(f(x)+f(y)\Big) - \Big(f(\overline{w}1)+f(\overline{w}0)\Big) &=& 3 \cdot 2^{n-2} + \tbinom{n-1}{i}-\tbinom{n-3}{i-3}+2\tbinom{n-2}{i-2}-\tbinom{n-2}{i-1}-\tbinom{n-2}{i}\\ 
&=& 3 \cdot 2^{n-2} + 2\tbinom{n-2}{i-2}-\tbinom{n-3}{i-3} \\ &=& 3 \cdot 2^{n-2} +\tbinom{n-2}{i-2}+\tbinom{n-3}{i-2} \\ &\leq& 3 \cdot 2^{n-2} + \tbinom{n-2}{\lceil (n-2)/2\rceil }+\tbinom{n-3}{\lceil (n-3)/2\rceil}.\end{eqnarray*}

So $\str_f(Q_{n}) \leq \str_f(Q_{n-2}) + 3 \cdot 2^{n-2} +\tbinom{n-2}{\lceil (n-2)/2\rceil }+\tbinom{n-3}{\lceil (n-3)/2\rceil}$.
\end{proof}

The next corollary follows directly from Theorems~\ref{thm:both0},~\ref{thm:both1odd},~\ref{thm:both11even} and Theorem~\ref{thm:both10even}.

\begin{corollary}\label{cor:recurrence} For $n \geq 5$, $$ \str(Q_{n}) \leq \str_f(Q_{n-2})+3 \cdot 2^{n-2} + \binom{n-3}{\big\lceil \tfrac{n-3}{2}\big\rceil}+\binom{n-2}{\big\lceil \tfrac{n-2}{2}\big\rceil}.$$\end{corollary}

\subsection{A comparison of upper bounds}

\begin{table}[h]
\[
\setlength\arraycolsep{3.5pt}
\begin{array}{|c|c|c|}
\hline
n & \text{upper bound on $\str(Q_n)$} & \text{upper bound on $\str(Q_n)$} \\
~ & \text{from Theorem 1.1} &  \text{from Corollary 2.9} \\
\hline
5 & 41 & 40 \\
6 & 81 & 78 \\
7 & 161 & 152 \\
8 & 321 & 300 \\
9 & 641 & 591 \\
10 & 1281 & 1173 \\
11 & 2561 & 2323 \\
12 & 5121 & 4623 \\
13 & 10241 & 9181 \\
 \hline
\end{array}
\]
\caption{Comparison of upper bounds for small values of $n$.}\label{smalln}\end{table}

Table~\ref{smalln} compares the upper bound given by Corollary~\ref{cor:recurrence} to the upper bound given by Theorem~\ref{Hyper}~\cite{I.2018} for $n \in \{5,6,\dots,13\}$.  For larger values of $n$, we compare the bounds by using the following inequality, due to~\cite{Wallis}, which holds for $k \geq 1$: \begin{equation}\label{ineq}\tbinom{2k}{k} < \tfrac{4^k}{\sqrt{\pi k}}.\end{equation} We further observe that if $k \geq 6$ then $\sqrt{\pi k} > 4$. From Inequality~(\ref{ineq}), it follows that for $k \geq 6$, \begin{equation}\label{ineq6} \tbinom{2k}{k} < 4^{k-1}.\end{equation}

\begin{corollary}\label{cor:last} For all $n \geq 14$, $\str(Q_{n}) \leq 2^{n}+2^{n-3}+28.$ \end{corollary}

\begin{proof} First, suppose $n-2$ is even and $n \geq 6$.  We use the fact that $\tbinom{2k-1}{k} = \tfrac{1}{2}\tbinom{2k}{k}$ and $\str(Q_4)=21$ to simplify the recurrence from Corollary~\ref{cor:recurrence} to that of Inequality~(\ref{eqnlist1}).
\begin{eqnarray} \str(Q_{n}) 
&\leq&\label{eqnlist1} 21+ 3 \Big(2^4+2^6+\cdots+2^{n-2}\Big)+\tfrac{3}{2} \Big(\tbinom{4}{2}+\tbinom{6}{3}+\cdots+\tbinom{(n-2)}{(n-2)/2}\Big) \\ 
&=& 21+3 \Big(\sum_{k=2}^{(n-2)/2} 4^k\Big)+\tfrac{3}{2} \Big(\sum_{k=2}^{(n-2)/2} \tbinom{2k}{k}\Big)\\ 
&=&\label{eqnlist9} 5+ 2^{n} +\tfrac{3}{2} \Big(\sum_{k=2}^{(n-2)/2} \tbinom{2k}{k}\Big)\\   
&=&\label{eqnlistH} 2^{n} + \tfrac{3}{2} \Big( \sum_{k=6}^{(n-2)/2} \tbinom{2k}{k}\Big) + 527\\
&<&\nonumber 2^{n} + \tfrac{3}{2} \Big( \sum_{k=6}^{(n-2)/2} 4^{k-1}\Big) + 527 ~~~~~~ \text{using~(\ref{ineq6})}\\  
&=&\nonumber 2^{n} + 2^{n-3}+15.\end{eqnarray}  Note that~(\ref{eqnlist1})-(\ref{eqnlist9}) hold for $n \geq 6$.  However, to get~(\ref{eqnlistH}), we split up a summation and this requires the condition that $n \geq 14$ (since the summation then runs from $k=6$ to $k=(n-2)/2$).

If $n$ is odd, then since $\str(Q_5)=40$,  a similar calculation shows that for odd $n \geq 15$, $\str(Q_{n}) \leq 2^{n}+2^{n-3}+28.$\end{proof}

Finally, we relate the bound in Corollary~\ref{cor:last} to the bound of Theorem~\ref{Hyper}~\cite{I.2018}.  
 
From Theorem~\ref{Hyper}~\cite{I.2018}, $\str(Q_n) \leq 2^n+2^{n-2}+1 = 2^n+2^{n-3}+28 +(2^{n-3}-27).$  Consequently, for $n \geq 14$, the bound from Corollary~\ref{cor:last} is an improvement of $(2^{n-3}-27)$ over Theorem~\ref{Hyper}~\cite{I.2018}.  Observe there remains room for improvement of the upper bound for $\str(Q_n)$: we used Inequality~(\ref{ineq6}), which holds for $k \geq 6$.  However, one could choose a larger value of $k$ to gain a different inequality; for example, if $k \geq 21$ then $\tbinom{2k}{k} < \frac{4^k}{8} = \frac{1}{2} \cdot 4^{k-1}$ and if $k \geq 82$ then $\tbinom{2k}{k} < \frac{4^k}{16} = 4^{k-2}$.  The tradeoff however, is that the constant term in the bound increases greatly. 

\medskip

\section*{Acknowledgements} 
Dylan Pearson acknowledges research support from the Atlantic Association for Research in the Mathematical Sciences (AARMS) and Mount Allison University. Melissa A. Huggan acknowledges research support from NSERC (grant application 2023-03395). M.E. Messinger acknowledges research support from NSERC (grant application 2018-04059).

\end{document}